\documentclass[10pt]{amsart}

\usepackage{graphicx}		
\usepackage{a4wide}			
							
\usepackage{amssymb}
\usepackage{amsmath}
\usepackage{mathtools}
\usepackage{amsbsy}
\usepackage{amsthm}
\usepackage{hyperref}
\usepackage{multirow}
\usepackage{array}
\newcounter{thm}[section]

\newtheorem{theorem}[thm]{Theorem}
\newtheorem{corollary}[thm]{Corollary}
\newtheorem{proposition}[thm]{Proposition}
\newtheorem{lemma}[thm]{Lemma}
\newtheorem*{lemma*}{Lemma}
\newtheorem{claim}[thm]{Claim}
\newtheorem*{claim*}{Claim}

\theoremstyle{definition}
\newtheorem{example}[thm]{Example}
\newtheorem*{notation}{Notation}

\newcommand{\eps}{\varepsilon}

\DeclareMathOperator{\Var}{Var}

\DeclarePairedDelimiter\abs{\lvert}{\rvert}

\newcommand{\NN}{\mathbb{N}}
\newcommand{\RR}{\mathbb{R}}

\newcommand{\PP}{\mathbb{P}}
\newcommand{\E}{\mathbb{E}}

\begin{document}

\title[Markov Chains for Integer Complexity]{An Application of Markov Chain Analysis \\to Integer Complexity}
\author{Christopher E. Shriver}
\address{UCLA Department of Mathematics, Los Angeles, CA}
\email{cshriver@math.ucla.edu}

\begin{abstract} The complexity $f(n)$ of an integer was introduced in 1953 by Mahler \& Popken: it is defined as the smallest number of $1$'s needed in conjunction with arbitrarily many +, * and parentheses to write an integer $n$ (for example, $f(6) \leq 5$ since $6 = (1+1)(1+1+1)$). The best known bounds are
$$ 3 \log_{3}{n} \leq f(n) \leq  3.635 \log_{3}{n}.$$
The lower bound is due to Selfridge (with equality for powers of 3); the upper bound was recently proven by Arias de Reyna \& Van de Lune, and holds on a set of natural density one. 

We use Markov chain methods to analyze a large class of algorithms, including one found by David Bevan that improves the upper bound to
$$ f(n) \leq  3.52 \log_{3}{n}$$
on a set of logarithmic density one.
\end{abstract}
\maketitle

\section{Introduction}
What is the minimum number of 1's needed to write an integer $n$, using only $+$, $*$ and parentheses? In this paper we denote by $f(n)$ the minimum number of ones needed to express $n$ in this manner.  For example, $f(6) = 5$ since $6$ can be written with five ones as $(1+1)(1+1+1)$, but not using four or fewer. Note that $f(11) \ne 2$; concatenation is not allowed. The problem dates back to a 1953 paper
of Mahler \& Popken \cite{mahler}. Guy popularized the problem in his survey \cite{guy1} and also included it in his book \textit{Unsolved Problems in Number Theory} \cite{guy2}. 
Recently, there has been renewed interest \cite{alt, alt2, alt1, reyna, latvia1, latvia2, steiner} in the problem -- the focus of our paper is on the asymptotic size of $f(n)$.

An interesting expression for $f(n)$ given by Guy in \cite{guy1} is
	\[ f(n) = \min_{\substack{ d | n \\ 2 \leq d \leq \sqrt{n} \\ 1 \leq e \leq n/2}} \big\{ f(d) + f(n/d),\  f(e) + f(n-e) \big\}, \]
which shows that the problem can be seen as arising from the difficulty of mixing additive and multiplicative behavior.
Several initial results are given in Guy's article \cite{guy1}. One of these, due to John Selfridge, is a lower bound of
	\[ f(n) \geq 3 \log_3 n, \]
with equality exactly when $n$ is a power of 3. As the author is not aware of any published proof of this fact, one is given below in Appendix \ref{sec:guyproof}.
Guy also describes several ways of obtaining an upper bound. A simple upper bound, resulting from expanding $n$ in base 2 and using Horner's scheme, yields
	\[ f(n) \leq 3 \log_2 n < 4.755 \log_3 n. \]
Unconditional improvement seems unlikely, since there exist numbers requiring an atypically large number of 1's to be represent: for example (taken from \cite{latvia1}),
$$ f(1439) = 3.928 \log_{3}{(1439)}.$$
Much recent focus has been on improvements for subsets of density 1. Returning to Guy's original argument, we see that a `typical' number $n \in \mathbb{N}$ should
have roughly the same number of 0's and 1's in its binary expansion, which suggests the upper bound 
	\[ f(n) \approx \frac{5}{2} \log_2 n < 3.963 \log_3 n \]
in the `generic' case (that is, numbers for which the proportion of 0's and 1's in the binary expansion is not too atypical). It is not difficult to show that the `generic' case gives rise to a subset of density 1 on which the bound holds. Guy also cites improvements by Isbell, the best of which is achieved by performing the same type of analysis in base 24, giving
	\[ f(n) < 3.817 \log_3 n \]
on a set of density 1.

The current best result is due to Fuller, Arias de Reyna, and Van de Lune \cite{reyna} and is based on similar considerations carried out in base $3359232.$

\begin{theorem}[Fuller, Arias de Reyna, and Van de Lune, 2014]
The set of natural numbers $n$ for which
	\[ f(n) \leq 3.635 \log_3 n \]
has density 1. 

\end{theorem}

Experimental results suggest that this can be improved: Iraids, Balodis, \u{C}er\c{n}enoks, Opmanis, Opmanis, and Podnieks \cite{latvia2} have conjectured, based on extensive numerical computation, that
	\[ \limsup_{n \to \infty} \frac{f(n)}{\log_3 n} \leq 3.37. \]

We emphasize that although there are many other problems related to integer complexity, our focus will be on the asymptotic size. We refer to the bibliography for further details.

\section{Main results}
\subsection{New upper bound.} The main subject of this paper is the analysis of algorithms generalizing the one used by Steinerberger \cite{steiner}.
The algorithm used in \cite{steiner} is a refinement of the methods discussed above and consists of a simultaneous consideration of numbers in bases 2 and 3. More precisely,
the algorithm proceeds as follows:\\

\textbf{Algorithm} (Greedy, base 6)\textbf{.}

\begin{enumerate}
	\item Take an arbitrary natural number $n$. If $n \leq 5$, use one of the optimal representations
	$1 = 1$, 
	$2 = 1 + 1$, 
	$3 = 1 + 1 + 1$, 
	$4 = 1 + 1 + 1 + 1$, or 
	$5 = 1 + 1 + 1 + 1 + 1$. Otherwise, move to step 2.
	\item Choose a representation depending on $n \pmod{6}$:
	\begin{itemize}
		\setlength{\itemsep}{0cm}
		\setlength{\parskip}{0cm}
	    	\item If $n \pmod{6} \equiv 0$, write as $3[n/3]$.
    		\item If $n \pmod{6} \equiv 1$, write as $3\big[(n-1)/3\big] + 1$.
 	   	\item If $n \pmod{6} \equiv 2$, write as $2[n/2]$.
 	   	\item If $n \pmod{6} \equiv 3$, write as $3[n/3]$.
 	   	\item If $n \pmod{6} \equiv 4$, write as $2[n/2]$.
 	   	\item If $n \pmod{6} \equiv 5$, write as $2\big[(n-1)/2\big]+1$.
	\end{itemize}
	Then apply step 2 to the result (the number in brackets) until one of the representations in step 1 can be applied.
\item Replace every 2 by (1+1) and every 3 by (1+1+1).
\end{enumerate}
Steinerberger proposed an argument that
	\[ f(n) \leq 3.66 \log_3 n \qquad \mbox{for `generic' $n$.} \]
	
This notion of `generic' is weaker than that of a density one subsequence. A previous version of this paper expanded Steinerberger's methods to substantially improve the result using the same notion of `generic,' but Juan Arias de Reyna and David Bevan pointed out a flaw in a preprint which was also present in Steinerberger's original paper \cite{steiner}.

The current version corrects the error and strengthens the claimed result, showing that the class of algorithms produces bounds that hold on sets of logarithmic density one. David Bevan also discovered a slightly better algorithm than the one included in the earlier preprint, giving the following:

\begin{theorem}[Main result]
\label{maintheorem}
The set of natural numbers $n$ such that
	\[ f(n) < 3.52 \log_3 n \]
has logarithmic density one.
\end{theorem}

The logarithmic density of a set $B$ of natural numbers is defined by
	\[ \lim_{N \to \infty} \frac{1}{\log n}\sum_{\substack{ n \in B \\[0.1em] n \leq N}} \frac{1}{n}, \]
if the limit exists. This is very closely related to the natural density, which is slightly stronger: if the natural density exists, then so does the logarithmic density, and the two are equal. Having calculated the logarithmic density, it would therefore be sufficient to prove that the natural density exists in order to know its value.

The proof of the theorem uses a Markov chain model that allows us to study
more effective algorithms related to Steinerberger's, answering an open problem stated in \cite{steiner}.

Table \ref{tab:thmcomp} shows the improvement of the new algorithm over the base 6 greedy algorithm and how well it fits the bound of Theorem \ref{maintheorem} for a few `arbitrarily' chosen large numbers.

\begin{table}[h!]
\begin{center}
\begin{tabular}{|c||c|c|c|}
\hline
\multirow{2}{*}{$n$} & \multicolumn{2}{c|}{number of ones used to represent $n$ by} & \multirow{2}{*}{$\lfloor 3.52 \log_3 n \rfloor$} \\
\cline{2-3}
& base 6 greedy & algorithm of Theorem \ref{maintheorem}\rule{0cm}{1em} & \\
\hline
$\lfloor \pi \cdot 10^{100} \rfloor + 10^{1000}$ & 7588 & 7372 & 7377 \rule{0cm}{1.1em}\\[0.2em]
$\lfloor \sqrt{2} \cdot 10^{100} \rfloor + 10^{2000}$ & 15232 & 14718 & 14755 \\[0.2em]
$\lfloor e \cdot 10^{100} \rfloor + 10^{3000}$ & 22823 & 22083 & 22132 \\[0.2em]
\hline
\end{tabular}
\end{center}
\caption{Comparison of performance of current and previous best known algorithms of the type considered in this paper.}
\label{tab:thmcomp}
\end{table}%

\subsection{Limitations of these methods.}
It seems worthwhile to note that the only limit to the methods discussed above is of a computational nature:
If, for some constant $3 \leq c < \infty,$
$$ \limsup_{n \to \infty} \frac{f(n)}{\log_3 n} \leq c$$
then both Guy's basis method as well as the Markov chain method will be able to prove the upper bound $c+ \varepsilon$ for every $\varepsilon > 0$. This can be seen by using an estimate of the type
$$ \frac{1}{b} \sum_{n=1}^{b}{f(n)} \lesssim \frac{1}{b} \sum_{n=1}^{b}{c \log_{3}{n}} \sim c \log_{3}{b}$$
to guarantee that a typical `digit' in a base $b$ representation requires at most $c \log_{3}{b}$ times the digit 1's. At the same time, to represent an integer $N$, we need
$$ \log_b{N} = \frac{\log{N}}{\log{b}} = \frac{\log_{3}{N}}{\log_3{b}} \qquad \mbox{digits in base}~b$$
and therefore a total number of
$$  \frac{\log_{3}{N}}{\log_3{b}}  c \log_{3}{b} = c \log_{3}{N} \qquad \mbox{ones.}$$
This implies that the Guy basis algorithm will eventually be effective. However, since the Markov chain algorithm is a generalization of the Guy basis algorithm, it will also
be able to show the same bound, likely with less computational effort. More sophisticated optimization techniques may be able to make significant improvements.

\section{A Review of the Markov Chain Method}
\subsection{Idea.} This section is primarily dedicated to an exposition of Steinerberger's approach, which did not receive a detailed explanation in the original paper. It is based on the method of proving Guy's upper bound of $4.755 \log_3 n$ by writing
	\[ n = a_0 + (1+1) (a_1 + (1+1) (a_2 + ( \cdots (1+1)(a_k + 1 + 1) \cdots ))) \]
with $a_i \in \{0,1\}$ (this is the expression of $n$ in terms of its base 2 digits using Horner's scheme). To emphasize the similarity, we can construct this expression algorithmically as follows: Given a number $n$,
\begin{enumerate}
	\item If $n =1$, the optimal representation is trivial. Otherwise move to step 2.
	\item Calculate $n \bmod 2$, then:
	\begin{itemize}
		\item If $n \bmod{2} = 0$, write as $2 [n/2]$
		\item If $n \bmod{2} = 1$, write as $2 [(n-1)/2] + 1$.
	\end{itemize}
	Repeat this step on the result (i.e.~the number in brackets) until the result is 1.
\item Replace every 2 by (1+1).
\end{enumerate}
Steinerberger's improvement was to look at the residue of $n$ modulo 6 and, depending on the residue, divide by either 2 or 3. As an example, given a number of the form $6n +5$ one could write
\begin{align*}
6n + 5 &= 1 + 2(3n+2) = 1 + (1+1)(3n + 2) \\
6n + 5 &= 2 + 3(2n + 1) = (1+1)+(1+1+1)(2n+1).
\end{align*}
The first representation provides a number that is half the size of its input at the cost of 3 ones; the second gives a number that is a third of its input size at the cost of 5 ones. Clearly, we want to decrease the number by a large factor, but we want to avoid using more ones than necessary. We now discuss a way of comparing those two options to decide on which may be favorable.

\subsection{Measuring inefficiency}
\label{sec:heuristic} This section describes the heuristic used in \cite{steiner}. We wish to emphasize that this heuristic will \textit{not} always lead to optimal results (this answers a question posed at the end of \cite{steiner}).
Suppose that we write $n$ as
	\[ n = \lambda \cdot \left[\frac{n-r}{\lambda}\right] + r \]
at a cost of $m$ ones. The inefficiency of this representation is defined to be
	\[ I = m - 3 \log_3 \lambda. \]
We remark that this quantity is never negative:
the cost $m$ is the number of ones it takes to write $b$ and $r$, so
	\[ m = f(\lambda) + f(r) \geq f(\lambda) \geq 3 \log_3 \lambda \]
and therefore
	\[ I \geq 0. \]

\subsection{The algorithm}
\label{sec:steineralg}
The algorithm given in \cite{steiner} is now produced by, for each residue class modulo 6, calculating the inefficiency of the two representations arising from dividing by either 2 or 3 and using the representation with the smaller inefficiency.
Explicit values for the inefficiencies are given in Table \ref{table_mod6}.\\

\textit{Local-global relations.} It is natural to ask whether such a local greedy selection procedure can give rise to optimal algorithms. Fixing a greedy selection at every residue
class gives rise to a unique global algorithm -- the behavior of the global algorithm depends nonlinearly on local steps (i.e. greedy local steps may create an overall
global dynamics that stays away from `effective' residue classes). The question of whether optimal local steps necessarily imply optimal global steps was posed
as an open problem in \cite{steiner}. We will answer this question in the negative (see Section 5.2).

\begin{table}[h!]
\begin{center}
{
\begin{tabular}{|c||c|c|c||c|c|c|}
	\hline
	\multirow{2}{*}{$n \bmod 6$} & \multicolumn{3}{|c||}{divide by 2} & \multicolumn{3}{|c|}{divide by 3} \\
	\cline{2-7}
	& representation & cost & inefficiency & representation & cost & inefficiency \\
	\hline
	0 & $2 \Big(\frac{n}{2}\Big)$		& 2 & 0.107 & $3 \Big(\frac{n}{3} \Big)$		& 3 & 0 \rule{0pt}{1.5em}\\[0.7em]
	1 & $2 \Big(\frac{n-1}{2}\Big) + 1$	& 3 & 1.107 & $3 \Big(\frac{n-1}{3} \Big) + 1$	& 4 & 1 \\[0.7em]
	2 & $2 \Big(\frac{n}{2}\Big)$		& 2 & 0.107 & $3 \Big(\frac{n-2}{3} \Big) + 2$	& 5 & 2 \\[0.7em]
	3 & $2 \Big(\frac{n-1}{2}\Big) + 1$	& 3 & 1.107 & $3 \Big(\frac{n}{3} \Big)$		& 3 & 0 \\[0.7em]
	4 & $2 \Big(\frac{n}{2}\Big)$		& 2 & 0.107 & $3 \Big(\frac{n-1}{3} \Big) + 1$ 	& 4 & 1 \\[0.7em]
	5 & $2 \Big(\frac{n-1}{2}\Big) + 1$	& 3 & 1.107 & $3 \Big(\frac{n-2}{3} \Big) + 2$	& 5 & 2 \\[0.7em]
	\hline
\end{tabular}
}
\end{center}
\caption{Calculation of inefficiencies for each residue class mod 6}
\label{table_mod6}
\end{table}

\subsection{Markov chain motivation} The key step in the analysis of an algorithm of such a type is to interpret its large-scale behavior as a Markov chain.
A Markov chain is a sequence of random variables $\mathbf{\Phi} = \{ \Phi_0, \Phi_1, \Phi_2, \ldots \}$ on a measurable space $X$, such that the distribution of $\Phi_i$ is dependent only on the distribution of $\Phi_{i-1}$. For our purposes, $X$ will be finite and the dependence will be described by a matrix $M$ via the relation
	\[ \pi_i = M \pi_{i-1} \]
where $\pi_i$ is the vector giving the probability distribution of $\Phi_i$ (that is, $\pi_i(j)$ is the probability that $\Phi_i = j$). We denote by $P^m(x,A)$ the probability that $x$ lands in the set $A$ after $m$ iterations:
	\[ P^m(x,A) := \PP \big[ \Phi_{i+m} \in A \ | \ \Phi_i = x \big]. \]

To model the algorithm described above, we let $X = \{0,1,2,3,4,5\}$ be the set of residues modulo 6. The algorithm naturally provides a sequence of states in $X$: given an integer of the form $6n + 1$, for example, the algorithm proceeds
by writing
$$ 6n +1 = 1 + (1+1+1)(2n).$$
A number of the form $2n$ has residue class $0, 2$ or $4$ when re-interpreted modulo 6. This means that the state 1 can be followed by 0, 2 or 4. We can similarly calculate possible successors of each state in $X$.
Possible paths are represented by the directed graph in Figure \ref{fig:graph}.\\

\begin{figure}[h!]
\includegraphics[width=0.5\textwidth]{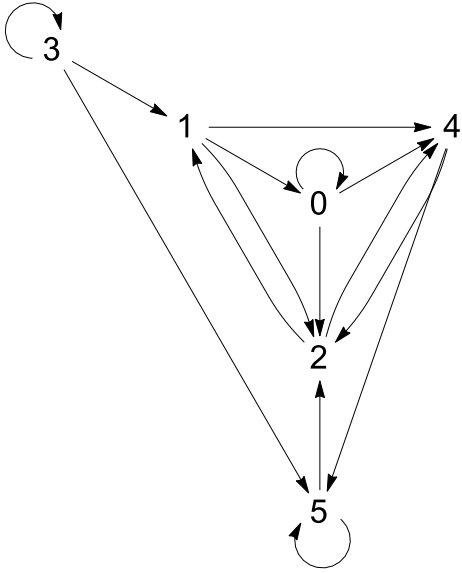}
\caption{Directed graph representing the possible paths of the algorithm through residue classes mod 6.}
\label{fig:graph}
\end{figure}

Since we are considering walks on a graph it is natural to seek a Markov chain model, but it is not immediately clear how to choose transition probabilities. This difficulty, which led to an error in Steinerberger's paper \cite{steiner} and a previous version of the current paper, will be addressed in the following discussion.

\section{Markov Chain Analysis}

\subsection{Asymptotic variance.} We start by citing a theorem that allows us to control the typical variance of functions of sample paths.

\begin{lemma}
\label{lem:clt}
	Let $\mathbf{\Phi}$ be a Markov chain on a finite space $X$. If there exists a state $i \in X$ and some $m \in \NN$ such that $i$ can be reached in exactly $m$ steps from any starting point with positive probability, that is,
$$P^m(x,i) > 0 \quad \mbox{ for all} \quad x \in X,$$
then any starting distribution approaches the unique stationary distribution $\pi$ at a geometric rate.

Moreover, we have the following: let $g:X \to \RR$ be bounded and 
	\[ S_n(g) = \sum_{k=1}^n g(\Phi_k). \]
Then if we let $\bar{g} = g - \int g\, d\pi$ where $\pi$ is the limit distribution, there exists a constant $0 \leq \gamma_g^2 < \infty$ such that
	\[ \lim_{n \to \infty} \Var \left[ \frac{1}{\sqrt{n}} S_n(\bar{g}) \right] = \gamma_g^2. \]

\end{lemma}

The proof involves checking a few criteria given in \cite{meyn}, which is done in Appendix \ref{pf:clt}. \\

For the remainder of this paper we will denote by $\E_\pi[g] \coloneqq \int g\, d\pi$ the expectation of a function $g$ under the distribution $\pi$. \\

This lemma will be used to show that a class of algorithms generalizing the one in Section \ref{sec:steineralg} can, with one additional condition, be used to bound the variance of the outcome of the algorithm. This bound will then be used to prove the main theorem. \\

\subsection{Generalized algorithms}

We consider algorithms $A$ defined by a map $\Lambda \colon \{0,1,\ldots,a-1\} \to \NN$ assigning dividing factors to residue classes mod $a$; we require that each factor $\Lambda(i)$ divides $a$ and that each factor is at least 2 (i.e. $\Lambda(i) \geq 2$ for all $0 \leq i \leq a-1$) to exclude the trivial step of dividing by 1. A map $\Lambda$ produces a representation for a natural number $n$ according to the following algorithm: \\

\begin{quote}
	\textbf{General algorithm.} Define $n_0 = n$. Let $r_0 = n_0 \bmod a$, $\lambda_0 = \Lambda(r_0)$, and $s_0 = n_0 \bmod \lambda_0$. We then have
		\[ n_0 = \lambda_0 \big[ (n_0 - s_0) / \lambda_0 \big] + s_0. \]
	Iteratively define $n_{i+1} = (n_i - s_i) / \lambda_i$, with $\lambda_i$ and $s_i$ defined similarly to $\lambda_0,s_0$ above, stopping when $n_k = 1$. Then we can write
		\[ n = \lambda_0 \big( \lambda_1 \big( \lambda_2 (\cdots (\lambda_{k-1} + s_{k-1}) \cdots) + s_2 \big) + s_1 \big) + s_0 \]
	and expressing the $\lambda_i$ and $s_i$ using their optimal representations gives a representation for $n$. (Note that $\lambda_i, s_i$ are natural numbers which are strictly less than $a$, so only finitely many optimal representations must be known to give this representation.)
\end{quote}
We say that $A$ is a base $a$ algorithm, since its action at each step depends on residue classes modulo $a$.

\begin{notation}
	We denote by ``$n \bmod a$'' the remainder upon dividing $n$ by $a$. Equivalence modulo $a$ will be denoted by $r \equiv n \pmod{a}$.
\end{notation}

\begin{example}
If we write $\Lambda$ as an $a$-tuple, Guy's algorithm would be represented as $(2, 2)$ and the greedy algorithm used by Steinerberger as $(3, 3, 2, 3, 2, 2)$. A greedy algorithm can be produced for any base $a$ using the heuristic in Section \ref{sec:heuristic} by choosing the dividing factor $\Lambda(i)$ that gives the lowest inefficiency for each $0 \leq i \leq a-1$.\\
\end{example}

\subsection{Representing natural numbers}
In order to analyze the behavior of an algorithm $A$, in this section we establish a precise connection between the natural numbers and a walks on directed graph $G(A)$ encoding the structure of $A$. \\

By considering which residue classes can be sent to which other residue classes by one step of the algorithm, we can construct a directed graph $G(A)$ with vertices $\{0, 1, 2, \ldots, a-1\}$ and with an edge $(r_0, r_1)$ exactly when there exists some $n_0 \equiv r_0 \pmod{a}$ such that one iteration of the algorithm applied to $n_0$ gives a number $n_1$ with $n_1 \equiv r_1 \pmod{a}$. For example, in Steinerberger's base 6 greedy algorithm, $9 = 3 \pmod{6}$ is sent to $9/3 = 3$, so there is an edge $(3,3)$.\\

Applying an algorithm $A$ to a number $n$ gives a sequence of residues, denoted by
	\[ R(n) \coloneqq \{r_0, r_1, \ldots, r_{k}\}, \]
with the $r_i$ as above (here and in the following discussion, the dependence on $A$ is implicit). Conversely, given a sequence of residues we can calculate $\lambda_i = \Lambda(r_i)$ and $s_i = n_i \bmod \lambda_i = r_i \bmod \lambda_i$ -- the last equality follows from the fact that each $\lambda_i$ divides $a$ and from the definition $r_i = n_i \bmod a$. Then we can produce a number using a representation similar to above:
	\[ N(\{r_0, r_1, \ldots, r_k\}) \coloneqq \lambda_0 \big( \lambda_1 \big( \lambda_2 (\cdots (\lambda_{k-1} + s_{k-1}) \cdots) + s_2 \big) + s_1 \big) + s_0. \]
It is clear from this definition that if $R(n)$ is the sequence of residues arising from applying the algorithm to $n$, then
	\[ n = N(R(n)) . \]
Hence $\NN$ is in bijection with the set $\{ R(n) \mid n \in \NN \}$. Notice that every $R(n)$, interpreted as a sequence of vertices of $G(A)$, gives a walk on the graph. The converse is also true:

\begin{claim}
\label{claim:R_onto}
	For every walk $\{v_0,v_1,\ldots,v_k\}$ on the graph $G(A)$ ending at the vertex $1$ (that is, $v_k=1)$, there exists some natural number $n$ such that $R(n) = \{v_0,v_1, \ldots, v_{k}\}$.
\end{claim}
\begin{proof}
	See Appendix \ref{sec:R_onto_pf}.
\end{proof}

Combining this result with the preceding discussion yields the following:

\begin{proposition}
\label{prop:bij}
	For any algorithm $A$, the functions $R$ and $N$ give a bijection between the natural numbers and walks on the graph $G(A)$ ending at the vertex $1$.
\end{proposition}

Using this connection, we will study the behavior of an algorithm $A$ by studying random walks on $G(A)$.

\subsection{Markov chain model}
Given an algorithm $A$, we define
	\[ A_k = \big\{ n \in \NN \ :\  \abs{R(n)} = k+1  \big\}. \]
These are the natural numbers which $A$ takes $k$ steps to process. We will choose a walk $\Psi^k$ of length $k+1$ ending at 1, and hence a random number $n = N(\Psi^k) \in A_k$, according to the following process:

Given a base $a$ algorithm $A$, let $\mathbf{\Phi} = \{\Phi_0, \Phi_1, \ldots\}$ be the Markov chain on the space $X = \{0,1,\ldots,a-1\}$ describing a random walk starting on a vertex of $G(A)$ chosen uniformly at random, and let $M$ be its transition matrix. For each $x \in X$, let $S_j(x)$ be the set of all possible successors of $x$ from which the state 1 can be reached in exactly $j$ steps. Suppose that there is some natural number $m$ such that, for any $n \geq m$, the state $1$ can be reached from any other state in exactly $n$ steps (the existence of such an $m$ is a straightforward consequence of the convergence of the random walk process to a stationary distribution $\pi$, provided $\pi(1) \ne 0$).

The walk $\Psi^k$ consists of a collection of random variables $\{ \Psi_0^k, \Psi_1^k, \ldots, \Psi_k^k\}$, defined as follows: Let $\Psi_0^k$ be given the uniform distribution on $X$. The rest of the process is defined inductively: for each $0 < i \leq k$ let $\Psi_i^k$ depend only on $\Psi_{i-1}^k$, with
	\[ \PP \big[ \Psi_i^k = x \ \mid\ \Psi_{i-1}^k  = y \big] =
				\left\{\begin{array}{ll}
					\big| S_{k - i}(y) \big|^{-1}, & x \in S_{k-i}(y) \\
					0, & \text{else.}
				\end{array}\right.\]  
	Note that $\Psi_k^k = 1$ with probability one, so that $\Psi^k$ is in fact a walk of length $k$ ending at 1. Also note that the random variables $\{ \Psi_0^k, \Psi_1^k, \ldots, \Psi_{k - m}^k\}$ are jointly equidistributed with the random variables $\{\Phi_0, \Phi_1, \ldots, \Phi_{k-m}\}$ associated with the random walk on $G(A)$. Hence, aside from a sequence of final steps of fixed length $m$, $\Psi^k$ is essentially a random walk on $G(A)$. In particular, it is close enough to being a random walk that we can apply Lemma \ref{lem:clt}:
	
\begin{proposition}
\label{prop:variance}
	Suppose an algorithm $A$ yields a graph $G(A)$ such that the state 1 can be reached from any other state in a fixed number of steps. Then there exists a natural number $m$ such that, for $n>m$, the state 1 can be reached from any other state in exactly $n$ steps.
	
	Furthermore: for $k>m$, let $\Psi^k = \{\Psi_0^k, \Psi_1^k, \ldots, \Psi_k^k\}$ be the collection of random variables defined above, giving a walk of length $k+1$ ending at 1. If $g \colon X \to \RR$ is a bounded function and we define
		\[ \tilde{S}_k(g) \coloneqq \sum_{i = 0}^{k-1} g(\Psi_i^k), \]
	then 
		\[ \lim_{k \to \infty} \Var \left[ \frac{1}{\sqrt{k}} \tilde{S}_k(\bar{g}) \right] = \gamma_g^2, \]
	where $\gamma_g^2$ is as in Lemma \ref{lem:clt}.
\end{proposition}

\begin{proof}
Since the state 1 can be reached from any other in a fixed number of steps, we know that the random walk process is ergodic and in particular the state 1 is not transient. So the existence of such an $m$ is guaranteed.

The second part follows directly from Lemma \ref{lem:clt}, since the part of $\Psi^k$ which deviates from the random walk is insignificant for large $k$.
\end{proof}

\subsection{Application of Markov model}
Our goal is to bound $f(n)$ by estimating $f_A(n)$, the number of ones used by the algorithm $A$ to represent a randomly chosen $n = N(\Psi^k)$. We will estimate this using the ``average cost''
	\[ \bar{c} \coloneqq \frac{f_A(n)}{k}, \]
and the ``average dividing factor''
	\[ \bar{\lambda} \coloneqq n^{1/k}. \]
Note that $k$ is fixed and $n, \bar{c}, \bar{\lambda}$ are random variables.

With these definitions, we can write
	\[ \frac{f_A(n)}{\log_3 n} =  \frac{\bar{c}}{\log \bar{\lambda}} \log 3. \]
We will use Proposition \ref{prop:variance} to study $\bar{c}$ and $\log\bar{\lambda}$ when $k$ is large.
The average cost $\bar{c}$ can be written as
	\[ \bar{c} = \frac{1}{k} \sum_{i=0}^{k-1} C(\Psi_i^k) = \frac{1}{k}S_k(C) \]
where $C(i)$ is the number of ones used by one step of the algorithm applied to $i \in X$, that is, the total number of ones needed to represent $\Lambda(i)$ and $i \pmod{\Lambda(i)}$ in their optimal forms. Clearly $C$ is a bounded function, so we can apply Proposition \ref{prop:variance}. Note that
	\[ \frac{1}{k} S_k(\bar{C}) = \bar{c} - \E_\pi [C], \]
so by Chebyshev,
	\[ \PP \Big[ \big| \bar{c} - \E_\pi [C] \big| > \varepsilon \Big] \leq \frac{\Var(\frac{1}{\sqrt{k}}S_k(\bar{C}))}{k \varepsilon^2}. \]
Therefore if $\delta_1 > \gamma_{\bar{C}}^2/\varepsilon^2$ then, by Proposition \ref{prop:variance},
	\[ \PP \Big[ \big| \bar{c} - \E_\pi [C] \big| > \varepsilon \Big] \leq \frac{\delta_1}{k} \]
for large enough $k$.
\bigbreak
		
A similar approach gives an estimate of $\log\bar{\lambda}$ in terms of $n$: We can write $n$ as
\begin{align*}
	n &= \lambda_0 \big( \lambda_1 \big( \lambda_2 (\cdots (\lambda_{k-1} + s_{k-1}) \cdots) + s_2 \big) + s_1 \big) + s_0 \\[0.4em]
	 &= \lambda_0 \lambda_1 \cdots \lambda_{k-1} + \lambda_0 \cdots \lambda_{k-2} s_{k-1} + \ldots + \lambda_0 \lambda_1 s_2 + \lambda_0 s_1 + s_0 \\
	 &= \lambda_0 \lambda_1 \cdots \lambda_{k-1} \left( 1 + \frac{s_{k-1}}{\lambda_{k-1}} + \ldots + \frac{s_2}{\lambda_2 \cdots \lambda_{k-1}} + \frac{s_1}{\lambda_1 \cdots \lambda_{k-1}} + \frac{s_0}{\lambda_0 \cdots \lambda_{k-1}} \right).
\end{align*}
Then of course
\begin{equation}
\label{eqn:lowerbound}
	n \geq \lambda_0 \lambda_1 \cdots \lambda_{k-1},
\end{equation}
and since $\lambda_i \geq 2$ and $s_i \leq a-1$
	\[ n \leq \lambda_0 \lambda_1 \cdots \lambda_{k-1} \left( 1 + \frac{a-1}{2} + \frac{a-1}{2^2} + \ldots + \frac{a-1}{2^k} \right) <  \lambda_0 \lambda_1 \cdots \lambda_{k-1} \cdot a. \]
Taking logarithms and dividing by $k$, the last two inequalities combined give
	\[ \frac{1}{k} \sum_{i=0}^{k-1} \log \lambda_i \leq \frac{1}{k} \log n \leq \frac{1}{k} \sum_{i=0}^{k-1} \log \lambda_i + \frac{1}{k} \log a. \]
Therefore since $\log \bar\lambda = \frac{1}{k}\log n$, for large enough $k$
	\[ \log\bar{\lambda} - \frac{1}{k} \sum_{i=0}^{k-1} \log \lambda_i \]
can be made arbitrarily small for all $n \in A_k$.

But $\lambda_i = \Lambda(\Psi_i^k)$, so
	\[ \frac{1}{k} \sum_{i=0}^{k-1} \log \lambda_i = \frac{1}{k} S_k (\log\Lambda). \]
Just as before, $\Lambda$ is a bounded function on $X$ and
	\[ \frac{1}{k} S_k (\overline{\log\Lambda}) = \frac{1}{k} S_k (\log\Lambda) - \E_\pi [\log\Lambda], \]
so if $\delta_2 > \gamma_{\Lambda}^2/\varepsilon^2$ then
	\[ \PP \Big[\, \abs[\Big]{ \frac{1}{k}S_k (\log\Lambda) - \E_\pi [\log \Lambda] } > \varepsilon \Big] \leq \frac{\delta_2}{k} \]
for large enough $k$.
Thus for any $\eps > 0$ there exists $K$ so that for $k \geq K$
	\[ \abs[\big]{ \log\bar{\lambda} - \frac{1}{k} S_k (\log\Lambda) } < \frac{\varepsilon}{2} \quad \text{for all } n \in A_k \]
and
	\[ \PP \Big[ \,\abs[\Big]{ \frac{1}{k} S_k (\log \Lambda) - \E_\pi [\log\Lambda] } > \frac{\varepsilon}{2} \Big] < \frac{\delta_2}{k}, \]
so
	\[ \PP \Big[ \,\abs[\Big]{ \log\bar{\lambda} - \E_\pi [\log\Lambda] } > \varepsilon \Big] < \frac{\delta_2}{k}. \]
	
Recal that $\frac{f_A(n)}{\log_3 n} =  \frac{\bar{c}}{\log_3 \bar{\lambda}\rule{0pt}{0.7em}}$; by continuity, for any $\varepsilon>0$ there exist $\varepsilon_1, \varepsilon_2$ such that 
	\[ \abs[\Big]{ \frac{f_A(n)}{\log_3 n} - \frac{\E_\pi [C]}{\E_\pi [\log_3 \Lambda]} } < \varepsilon \]
whenever $\abs[\big]{\bar{c} - \E_\pi[C]} < \varepsilon_1$ and $\abs[\big]{ \log_3 \bar{\lambda} - \E_\pi [\log_3 \Lambda]} < \varepsilon_2$.

It follows that for any $\eps > 0$ there exists $K$ such that for $k \geq K$,
	\[ \PP \Big[ \,\abs[\Big]{ \frac{f_A(n)}{\log_3 n} - \frac{\E_\pi [C]}{\E_\pi [\log_3 \Lambda]} } > \varepsilon \Big] < \frac{\delta_1 + \delta_2}{k}. \]
In particular, if we let $\delta = \delta_1 + \delta_2$,
\begin{equation}
	\PP \Big[ f_A(n) > \Big( \frac{\E_\pi [C]}{\E_\pi [\log_3 \Lambda]} + \varepsilon \Big) \log_3 n \Big] < \frac{\delta}{k}.
\end{equation}

This provides a proof of the following:

\begin{proposition}
\label{prop:probbound}
	Given an algorithm $A$ satisfying the condition of Proposition \ref{prop:variance} and any $\varepsilon>0$, define
		\[ B_k \coloneqq \left\{ m \in A_k \ : \ f_A(m) > \Big( \frac{\E_\pi [C]}{\E_\pi [\log_3 \Lambda]} + \varepsilon \Big) \log_3 m \right\}. \]
	Then there exists $\delta>0$ such that, for large $k$, 
		\[ \PP [ N(\Psi^k) \in B_k ] < \frac{\delta}{k}. \] \\
\end{proposition}

\subsection{Logarithmic density}
We first note that, if an algorithm divides by $\lambda$ for some state, then that state has $\lambda$ possible successors.

Fix $b \in A_k$ and let $R(b) = \{r_0, r_1, \ldots, r_{k-1}, r_k=1\}$ as above and $\lambda_i = \Lambda(r_i)$ for each $0 \leq i < k$. Consider $\PP[n = b]$; by Proposition \ref{prop:bij}, this is the same as the probability that $\Psi^k$ is equal to the walk $R(b)$. We can calculate this using the definition of $\Psi^k$: The first residue is chosen `correctly' (that is, $\Psi_0^k=r_0$) with probability $1/a$; after that, $\Psi^k$ chooses uniformly from all successors which preserve the possibility of ending at 1. Let $p_i$ be the probability of making the right choice at the $i$th step given that all previous choices were correct; then $p_0 = a^{-1}$ and, for $i >0$,
	\[ p_i = \PP[ \Psi_{i}^k = r_i \ | \ \Psi_{i-1}^k = r_{i-1} ]. \]
Maintaining our assumption from Proposition \ref{prop:variance}, there is some $m$ independent of $k$ so that for $0 < i < k - m$ we have exactly $\lambda_{i-1}$ valid successors at the $i$th step and
	\[ p_i = (\lambda_{i-1})^{-1}. \]
For $i \geq k-m$, the choice is made uniformly from \emph{at most} $\lambda_{i-1}$ options (some successors may not be valid, since they might make it impossible to end at 1 at the $k$th step), so
	\[ p_i \geq (\lambda_{i-1})^{-1}. \]

Recall from Equation \ref{eqn:lowerbound} above that $b \geq \lambda_0 \lambda_1 \cdots \lambda_{k-1}$. So
\begin{align*}
	\frac{1}{b} & \leq \frac{1}{\lambda_0 \lambda_1 \cdots \lambda_{k-1}} \\
	&= p_1 p_2 \cdots p_{k-m} \frac{1}{\lambda_{k-m} \cdots \lambda_{k-1}} \\
	&\leq (a \cdot p_0) \cdot (p_1 p_2 \cdots p_{k-m}) \cdot (p_{k-m+1} \cdots p_k) \\
	&= a \cdot \PP[ n = b].
\end{align*}

Therefore
\begin{equation}
\label{eqn:Bkbound}
	\sum_{b \in B_k} \frac{1}{b} \leq a \sum_{b \in B_k} \PP [n = b] = a \cdot \PP [n \in B_k].
\end{equation}

We can now prove the following:
\begin{theorem}
	Let $A$ be an algorithm satisfying the condition of Proposition \ref{prop:variance}, and for any $\varepsilon>0$ let
		\[ B \coloneqq \left\{ b \in \NN \ :\ f_A(b) > \Big( \frac{\E_\pi [C]}{\E_\pi [\log_3 \Lambda]} + \varepsilon \Big) \log_3 b \right\}. \]
	Then $B$ has logarithmic density zero.
\end{theorem}

\begin{proof}
	Note that, since 2 is the smallest dividing factor used by any step, elements of $A_i$ can be no smaller than $2^i$. Therefore, if we let $2^m \leq M < 2^{m+1}$, we have
		 \[ \{1, 2, \ldots, M\} \subset \bigcup_{k=1}^m A_k. \]
	So, using Equation \ref{eqn:Bkbound} above,
		\[ \sum_{\substack{b \in B \\ b \leq M}} \frac{1}{b} \leq \sum_{k=1}^m \sum_{b \in B_k} \frac{1}{b} \leq a \cdot \sum_{k=1}^m  \PP[ N(\Psi^k) \in B_k]. \]
	From Proposition \ref{prop:probbound}, the summand is bounded above by $\frac{\delta}{k}$ for large $k$. Therefore there exists some constant $C$ such that
		\[ \sum_{\substack{b \in B \\ b \leq M}} \frac{1}{b} \leq a \cdot \sum_{k=1}^m \frac{\delta}{k} + C = a\delta \cdot H_m + C, \]
	where $H_m$ denotes the $m$th harmonic number.
	
	Now
		\[ \frac{1}{\log M}\sum_{\substack{b \in B \\ b \leq M}} \frac{1}{b} \leq \frac{a\delta \cdot H_m + C}{\log 2^m} = \frac{a\delta \cdot H_m + C}{m\log 2}. \]
	The right hand side approaches zero as $m$ approaches infinity, so the left hand side approaches zero as $M$ approaches infinity. Therefore $B$ has logarithmic density zero.
\end{proof}

Since $f \leq f_A$ for all algorithms $A$, we immediately get the following:

\begin{corollary}
\label{cor:logdensity}
	For any $\varepsilon>0$ let
		\[ B \coloneqq \left\{ b \in \NN \ :\ f(b) > \Big( \frac{\E_\pi [C]}{\E_\pi [\log_3 \Lambda]} + \varepsilon \Big) \log_3 b \right\}. \]
	Then $B$ has logarithmic density zero.
\end{corollary}

\section{Proof of Main Result}
\subsection{Greedy algorithm in higher basis.} 
The most obvious new class of algorithms to consider is greedy algorithms for higher bases. For example, the greedy algorithm in base 30 can be succinctly written as 
	\[ \Lambda = (3, 3, 2, 3, 2, 5, 3, 3, 2, 3, \ 2, 2, 3, 3, 2, 3, 2, 2, 3, 3, \ 2, 3, 2, 2, 3, 5, 2, 3, 2, 2), \]
giving rise to a transition matrix 
	\tiny\[  \text{\normalsize$M=$} \arraycolsep=1.5pt\left(\begin{array}{cccccccccccccccccccccccccccccc}
	\frac{1}{3} & \frac{1}{3} & \cdot & \cdot & \cdot & \cdot & \cdot & \cdot & \cdot & \cdot & \cdot & \cdot & \cdot & \cdot & \cdot & \cdot & \cdot & \cdot & \cdot & \cdot & \cdot & \cdot & \cdot & \cdot & \cdot & \cdot & \cdot & \cdot & \cdot & \cdot \\ 
	\cdot & \cdot & \frac{1}{2} & \frac{1}{3} & \cdot & \frac{1}{5} & \cdot & \cdot & \cdot & \cdot & \cdot & \cdot & \cdot & \cdot & \cdot & \cdot & \cdot & \cdot & \cdot & \cdot & \cdot & \cdot & \cdot & \cdot & \cdot & \cdot & \cdot & \cdot & \cdot & \cdot \\ 
	\cdot & \cdot & \cdot & \cdot & \frac{1}{2} & \cdot & \frac{1}{3} & \frac{1}{3} & \cdot & \cdot & \cdot & \cdot & \cdot & \cdot & \cdot & \cdot & \cdot & \cdot & \cdot & \cdot & \cdot & \cdot & \cdot & \cdot & \cdot & \cdot & \cdot & \cdot & \cdot & \cdot \\ 
	\cdot & \cdot & \cdot & \cdot & \cdot & \cdot & \cdot & \cdot & \cdot & \frac{1}{3} & \cdot & \cdot & \cdot & \cdot & \cdot & \cdot & \cdot & \cdot & \cdot & \cdot & \cdot & \cdot & \cdot & \cdot & \cdot & \cdot & \cdot & \cdot & \cdot & \cdot \\ 
	\cdot & \cdot & \cdot & \cdot & \cdot & \cdot & \cdot & \cdot & \frac{1}{2} & \cdot & \cdot & \cdot & \frac{1}{3} & \frac{1}{3} & \cdot & \cdot & \cdot & \cdot & \cdot & \cdot & \cdot & \cdot & \cdot & \cdot & \cdot & \cdot & \cdot & \cdot & \cdot & \cdot \\ 
	\cdot & \cdot & \cdot & \cdot & \cdot & \cdot & \cdot & \cdot & \cdot & \cdot & \frac{1}{2} & \frac{1}{2} & \cdot & \cdot & \cdot & \frac{1}{3} & \cdot & \cdot & \cdot & \cdot & \cdot & \cdot & \cdot & \cdot & \cdot & \frac{1}{5} & \cdot & \cdot & \cdot & \cdot \\ 
	\cdot & \cdot & \cdot & \cdot & \cdot & \cdot & \cdot & \cdot & \cdot & \cdot & \cdot & \cdot & \cdot & \cdot & \cdot & \cdot & \cdot & \cdot & \frac{1}{3} & \frac{1}{3} & \cdot & \cdot & \cdot & \cdot & \cdot & \cdot & \cdot & \cdot & \cdot & \cdot \\ 
	\cdot & \cdot & \cdot & \cdot & \cdot & \frac{1}{5} & \cdot & \cdot & \cdot & \cdot & \cdot & \cdot & \cdot & \cdot & \frac{1}{2} & \cdot & \cdot & \cdot & \cdot & \cdot & \cdot & \frac{1}{3} & \cdot & \cdot & \cdot & \cdot & \cdot & \cdot & \cdot & \cdot \\ 
	\cdot & \cdot & \cdot & \cdot & \cdot & \cdot & \cdot & \cdot & \cdot & \cdot & \cdot & \cdot & \cdot & \cdot & \cdot & \cdot & \frac{1}{2} & \frac{1}{2} & \cdot & \cdot & \cdot & \cdot & \cdot & \cdot & \frac{1}{3} & \cdot & \cdot & \cdot & \cdot & \cdot \\ 
	\cdot & \cdot & \cdot & \cdot & \cdot & \cdot & \cdot & \cdot & \cdot & \cdot & \cdot & \cdot & \cdot & \cdot & \cdot & \cdot & \cdot & \cdot & \cdot & \cdot & \cdot & \cdot & \cdot & \cdot & \cdot & \cdot & \cdot & \frac{1}{3} & \cdot & \cdot \\ 
	\frac{1}{3} & \frac{1}{3} & \cdot & \cdot & \cdot & \cdot & \cdot & \cdot & \cdot & \cdot & \cdot & \cdot & \cdot & \cdot & \cdot & \cdot & \cdot & \cdot & \cdot & \cdot & \frac{1}{2} & \cdot & \cdot & \cdot & \cdot & \cdot & \cdot & \cdot & \cdot & \cdot \\ 
	\cdot & \cdot & \cdot & \frac{1}{3} & \cdot & \cdot & \cdot & \cdot & \cdot & \cdot & \cdot & \cdot & \cdot & \cdot & \cdot & \cdot & \cdot & \cdot & \cdot & \cdot & \cdot & \cdot & \frac{1}{2} & \frac{1}{2} & \cdot & \frac{1}{5} & \cdot & \cdot & \cdot & \cdot \\ 
	\cdot & \cdot & \cdot & \cdot & \cdot & \cdot & \frac{1}{3} & \frac{1}{3} & \cdot & \cdot & \cdot & \cdot & \cdot & \cdot & \cdot & \cdot & \cdot & \cdot & \cdot & \cdot & \cdot & \cdot & \cdot & \cdot & \cdot & \cdot & \cdot & \cdot & \cdot & \cdot \\ 
	\cdot & \cdot & \cdot & \cdot & \cdot & \frac{1}{5} & \cdot & \cdot & \cdot & \frac{1}{3} & \cdot & \cdot & \cdot & \cdot & \cdot & \cdot & \cdot & \cdot & \cdot & \cdot & \cdot & \cdot & \cdot & \cdot & \cdot & \cdot & \frac{1}{2} & \cdot & \cdot & \cdot \\ 
	\cdot & \cdot & \cdot & \cdot & \cdot & \cdot & \cdot & \cdot & \cdot & \cdot & \cdot & \cdot & \frac{1}{3} & \frac{1}{3} & \cdot & \cdot & \cdot & \cdot & \cdot & \cdot & \cdot & \cdot & \cdot & \cdot & \cdot & \cdot & \cdot & \cdot & \frac{1}{2} & \frac{1}{2} \\ 
	\cdot & \cdot & \cdot & \cdot & \cdot & \cdot & \cdot & \cdot & \cdot & \cdot & \cdot & \cdot & \cdot & \cdot & \cdot & \frac{1}{3} & \cdot & \cdot & \cdot & \cdot & \cdot & \cdot & \cdot & \cdot & \cdot & \cdot & \cdot & \cdot & \cdot & \cdot \\ 
	\cdot & \cdot & \frac{1}{2} & \cdot & \cdot & \cdot & \cdot & \cdot & \cdot & \cdot & \cdot & \cdot & \cdot & \cdot & \cdot & \cdot & \cdot & \cdot & \frac{1}{3} & \frac{1}{3} & \cdot & \cdot & \cdot & \cdot & \cdot & \cdot & \cdot & \cdot & \cdot & \cdot \\ 
	\cdot & \cdot & \cdot & \cdot & \frac{1}{2} & \cdot & \cdot & \cdot & \cdot & \cdot & \cdot & \cdot & \cdot & \cdot & \cdot & \cdot & \cdot & \cdot & \cdot & \cdot & \cdot & \frac{1}{3} & \cdot & \cdot & \cdot & \frac{1}{5} & \cdot & \cdot & \cdot & \cdot \\ 
	\cdot & \cdot & \cdot & \cdot & \cdot & \cdot & \cdot & \cdot & \cdot & \cdot & \cdot & \cdot & \cdot & \cdot & \cdot & \cdot & \cdot & \cdot & \cdot & \cdot & \cdot & \cdot & \cdot & \cdot & \frac{1}{3} & \cdot & \cdot & \cdot & \cdot & \cdot \\ 
	\cdot & \cdot & \cdot & \cdot & \cdot & \frac{1}{5} & \cdot & \cdot & \frac{1}{2} & \cdot & \cdot & \cdot & \cdot & \cdot & \cdot & \cdot & \cdot & \cdot & \cdot & \cdot & \cdot & \cdot & \cdot & \cdot & \cdot & \cdot & \cdot & \frac{1}{3} & \cdot & \cdot \\ 
	\frac{1}{3} & \frac{1}{3} & \cdot & \cdot & \cdot & \cdot & \cdot & \cdot & \cdot & \cdot & \frac{1}{2} & \frac{1}{2} & \cdot & \cdot & \cdot & \cdot & \cdot & \cdot & \cdot & \cdot & \cdot & \cdot & \cdot & \cdot & \cdot & \cdot & \cdot & \cdot & \cdot & \cdot \\ 
	\cdot & \cdot & \cdot & \frac{1}{3} & \cdot & \cdot & \cdot & \cdot & \cdot & \cdot & \cdot & \cdot & \cdot & \cdot & \cdot & \cdot & \cdot & \cdot & \cdot & \cdot & \cdot & \cdot & \cdot & \cdot & \cdot & \cdot & \cdot & \cdot & \cdot & \cdot \\ 
	\cdot & \cdot & \cdot & \cdot & \cdot & \cdot & \frac{1}{3} & \frac{1}{3} & \cdot & \cdot & \cdot & \cdot & \cdot & \cdot & \frac{1}{2} & \cdot & \cdot & \cdot & \cdot & \cdot & \cdot & \cdot & \cdot & \cdot & \cdot & \cdot & \cdot & \cdot & \cdot & \cdot \\ 
	\cdot & \cdot & \cdot & \cdot & \cdot & \cdot & \cdot & \cdot & \cdot & \frac{1}{3} & \cdot & \cdot & \cdot & \cdot & \cdot & \cdot & \frac{1}{2} & \frac{1}{2} & \cdot & \cdot & \cdot & \cdot & \cdot & \cdot & \cdot & \frac{1}{5} & \cdot & \cdot & \cdot & \cdot \\ 
	\cdot & \cdot & \cdot & \cdot & \cdot & \cdot & \cdot & \cdot & \cdot & \cdot & \cdot & \cdot & \frac{1}{3} & \frac{1}{3} & \cdot & \cdot & \cdot & \cdot & \cdot & \cdot & \cdot & \cdot & \cdot & \cdot & \cdot & \cdot & \cdot & \cdot & \cdot & \cdot \\ 
	\cdot & \cdot & \cdot & \cdot & \cdot & \frac{1}{5} & \cdot & \cdot & \cdot & \cdot & \cdot & \cdot & \cdot & \cdot & \cdot & \frac{1}{3} & \cdot & \cdot & \cdot & \cdot & \frac{1}{2} & \cdot & \cdot & \cdot & \cdot & \cdot & \cdot & \cdot & \cdot & \cdot \\ 
	\cdot & \cdot & \cdot & \cdot & \cdot & \cdot & \cdot & \cdot & \cdot & \cdot & \cdot & \cdot & \cdot & \cdot & \cdot & \cdot & \cdot & \cdot & \frac{1}{3} & \frac{1}{3} & \cdot & \cdot & \frac{1}{2} & \frac{1}{2} & \cdot & \cdot & \cdot & \cdot & \cdot & \cdot \\ 
	\cdot & \cdot & \cdot & \cdot & \cdot & \cdot & \cdot & \cdot & \cdot & \cdot & \cdot & \cdot & \cdot & \cdot & \cdot & \cdot & \cdot & \cdot & \cdot & \cdot & \cdot & \frac{1}{3} & \cdot & \cdot & \cdot & \cdot & \cdot & \cdot & \cdot & \cdot \\ 
	\cdot & \cdot & \cdot & \cdot & \cdot & \cdot & \cdot & \cdot & \cdot & \cdot & \cdot & \cdot & \cdot & \cdot & \cdot & \cdot & \cdot & \cdot & \cdot & \cdot & \cdot & \cdot & \cdot & \cdot & \frac{1}{3} & \cdot & \frac{1}{2} & \cdot & \cdot & \cdot \\ 
	\cdot & \cdot & \cdot & \cdot & \cdot & \cdot & \cdot & \cdot & \cdot & \cdot & \cdot & \cdot & \cdot & \cdot & \cdot & \cdot & \cdot & \cdot & \cdot & \cdot & \cdot & \cdot & \cdot & \cdot & \cdot & \frac{1}{5} & \cdot & \frac{1}{3} & \frac{1}{2} & \frac{1}{2} \\ 
\end{array}\right)\text{\normalsize.} \]\normalsize

It is quite easy to derive a numerical approximation of the stationary distribution as
	\begin{align*}
		\pi = (&0.01605, 0.03211, 0.03913, 0, 0.03929, 0.06272, 0.01272, 0.04573, 0.03753, 0, 0.05072, 0.05584, \\
		&0.01948, 0.04210, 0.06637, 0, 0.03228, 0.02909, 0.00684, 0.03131, 0.06934, 0, 0.05267, 0.04013, \\
		&0.02053, 0.04721, 0.05912, 0, 0.03640, 0.05529)^T.
	 \end{align*}
Using the transition matrix, it can be checked that the state 1 can be reached from any other state in exactly 9 steps, so that Corollary \ref{cor:logdensity} applies. From this algorithm we get the improved result
that
	\[ f(n) \leq 3.59 \log n \]
on a set of logarithmic density one.

\subsection{Improving on the greedy algorithm}
It is possible to improve on the greedy algorithm (this was not clear a priori). We define a `landscape' of algorithms by associating to each algorithm $A$ of the type above
a cost
	\[ c_A =  \frac{\E_\pi [C]}{\E_\pi [\log_3 \Lambda ]} \]
and saying that two algorithms $A_1, A_2$ are adjacent if their tuples $\Lambda_1, \Lambda_2$ agree on all but exactly one entry. The problem of finding a better bound is now a nonlinear optimization problem on a large search space.
The best known algorithm found by the author was found using simulated annealing. It has base
$2 \cdot 3 \cdot 5 \cdot 7 \cdot 11 = 2310$ and constant $c \approx 3.5286$. This algorithm implies the result
	\[ f(n) < 3.53 \log_3 n \quad \text{on a set of logarithmic density one.} \]
David Bevan's improved algorithm uses the same base, but has a constant of $c \approx 3.5197$, giving
	\[ f(n) < 3.52 \log_3 n \quad \text{on a set of logarithmic density one.} \]
The corresponding 2310-tuple is given in Appendix \ref{app:alg}.
	
\begin{table}[h!]
\begin{center}
{
\begin{tabular}{|c||>{\centering\arraybackslash}p{7em}|>{\centering\arraybackslash}p{8em}|>{\centering\arraybackslash}p{9em}|}
	\hline
	\multirow{2}{*}{factor} & \multicolumn{3}{c|}{frequency of use} \\
	\cline{2-4}
	& base 6 greedy \newline ($c \approx 3.65$) & base 2310 greedy \newline ($c \approx 3.56$) & base 2310 improved \newline ($c \approx 3.53$) \\
	\hline
	2 & 0.769 & 0.628 & 0.507 \\
	3 & 0.231 & 0.213 & 0.352 \\
	4 & $\cdot$ & $\cdot$ & $\cdot$ \\
	5 & $\cdot$ & 0.102 & 0.063 \\
	6 & 0 & 0 & 0 \\
	7 & $\cdot$ & 0.058 & 0.057 \\
	8 & $\cdot$ & $\cdot$ & $\cdot$ \\
	9 & $\cdot$ & $\cdot$ & $\cdot$ \\
	10 & $\cdot$ & 0 & 0 \\
	11 & $\cdot$ & 0 & 0.021 \\
	\hline
\end{tabular}
\caption{Comparison of frequency of factors used by improved algorithm with greedy algorithms. The frequency is the proportion of states which use a given factor, weighted by the limit distribution. Dots represent factors not allowed given the base.}
\label{tab:factfreq}
}
\end{center}
\end{table}%

In comparison, the greedy algorithm for base 2310 gives a constant of $3.561$. The improvement can be seen to come in part from the increased frequency with which 3 is used as a dividing factor, as shown in Table \ref{tab:factfreq}. Note also that the improved algorithm uses 11 as a dividing factor one in every fifty steps, even though it is never locally optimal (as can be seen from the center column). It seems that by taking a few locally inefficient steps, the algorithm is able to divide by 3, which is optimal according to the greedy heuristic, in ten percent more steps. Bevan's further improvement also makes use of the dividing factor 55.

Calling this metric `frequency of use' is perhaps not completely honest, since the stationary distribution gives the percentage of time the random walk on $G(A)$ spends in each state rather than the algorithm $A$ itself; we note, however, that the bound produced by an algorithm is dependent only on this stationary distribution $\pi$ and not on the `actual' behavior of the algorithm.

\section{Concluding Remarks}
\subsection{Extendability.} Using higher bases and running simulated annealing for more iterations will produce slightly better bounds. The bound produced by an algorithm is limited by the complexity of its dividing factors, and since most small numbers have suboptimal complexity (that is, $f(n)/\log_3 n$ is well above $3$), very good algorithms will have to include large dividing factors. But large factors will be used infrequently by an optimized algorithm, since dividing out by any factor is only worth it if the number of ones it costs to remove the remainder is small. This happens infrequently for large numbers, so including a single large factor, no matter how efficient it is by itself, will not have much of an effect on the resulting bound. It therefore seems unlikely that significant improvement could be made by finding better algorithms of the form considered here (which encompasses Guy's method as a special case) without using very large bases. Still, we consider the question of how to optimize over the space of algorithms a fascinating problem.

\subsection{Natural density.}  It would also be interesting to strengthen Theorem \ref{maintheorem} to state that $f(n) < 3.52 \log_3 n$ for a set of integers of \emph{natural} density 1, with the same holding 
for all bounds produced by the same type of algorithm. Since the logarithmic density has been established, it would be enough to prove the existence of a natural density for the sets $B$.

\subsection{Brownian motion} This comment, suggested by Stefan Steinerberger, is inspired by work of Sinai on large-scale properties of the $3x+1$ problem. Define $T:\mathbb{N} \rightarrow \mathbb{N}$ via
$$ T(n) = \begin{cases} n/2 \qquad &\mbox{if}~n~\mbox{is even.} \\
\frac{3n + 1}{2} \qquad &\mbox{if}~n~\mbox{is odd.}
\end{cases}$$
The $3x+1$ problem asks whether for every $n \in \mathbb{N}$ there is some $k$  such that repeated iteration yields $T^{(k)}(n) = 1$.
If one were to assume that being even or odd is equally likely for numbers arising in the course of these iterations, then one would expect an average decay of
$$ T^{(k)}(n) \sim \left(\frac{3}{4}\right)^{\frac{k}{2}} n \quad \mbox{and that} \quad  T^{(k)}(n) \left(\frac{4}{3}\right)^\frac{k}{2} \quad \mbox{appears to be random on a logarithmic scale.}$$
Some rigorous results in that direction have been obtained by Sinai \cite{sinai1,sinai2,sinai3}. Since our class of algorithms is related to the $3x+1$ problem (the only but admittedly crucial 
difference being that our algorithms always decrease the input), we were motivated to see whether similar phenomena appear. Indeed, it seems that a suitable rescaled series
of numbers behaves like Brownian motion on a logarithmic scale (see Figure \ref{fig:brownian}). We consider this to be a promising direction for further research.

\begin{figure}[h]
	\includegraphics[width=0.49\textwidth]{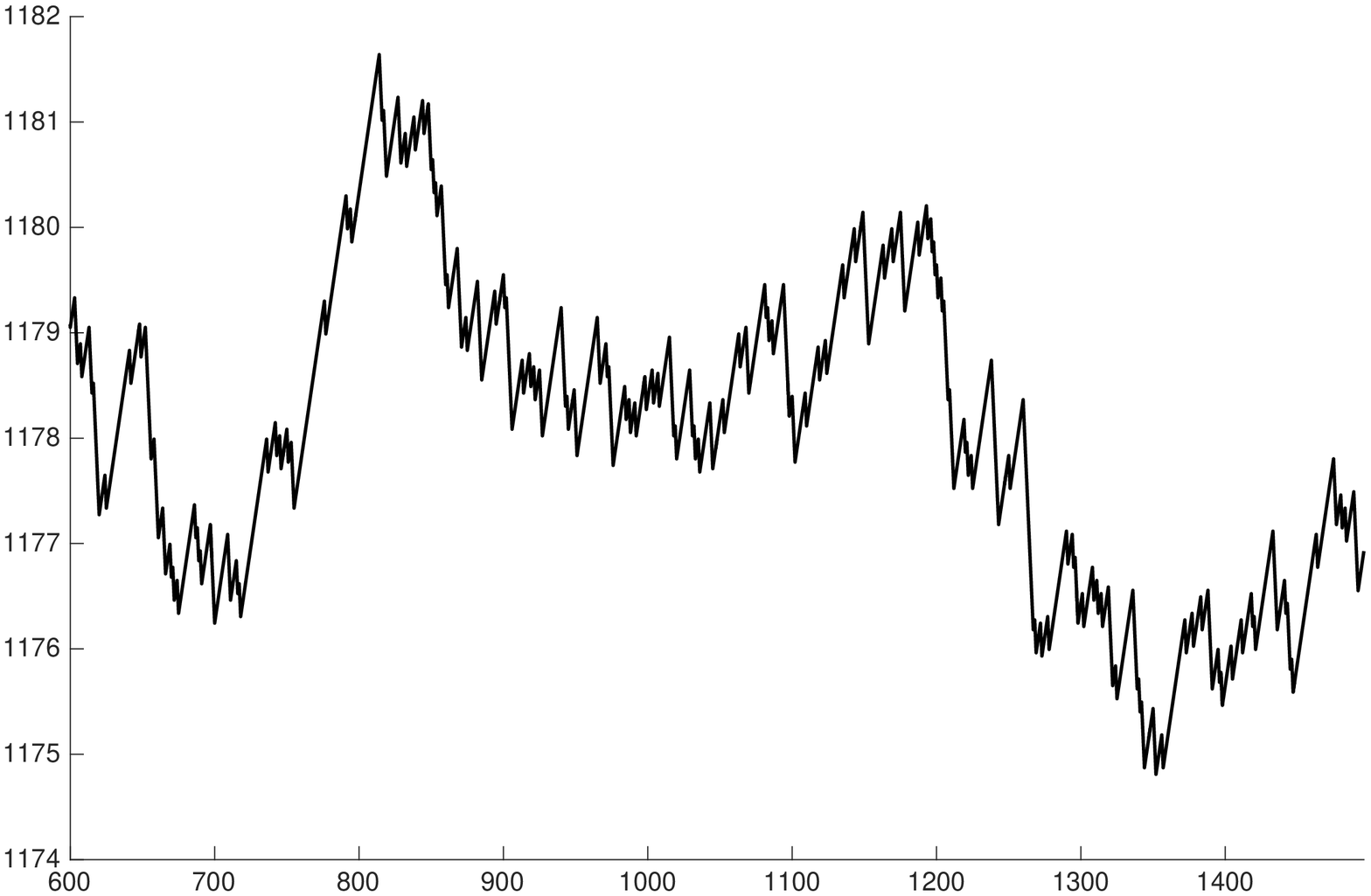}
	\includegraphics[width=0.49\textwidth]{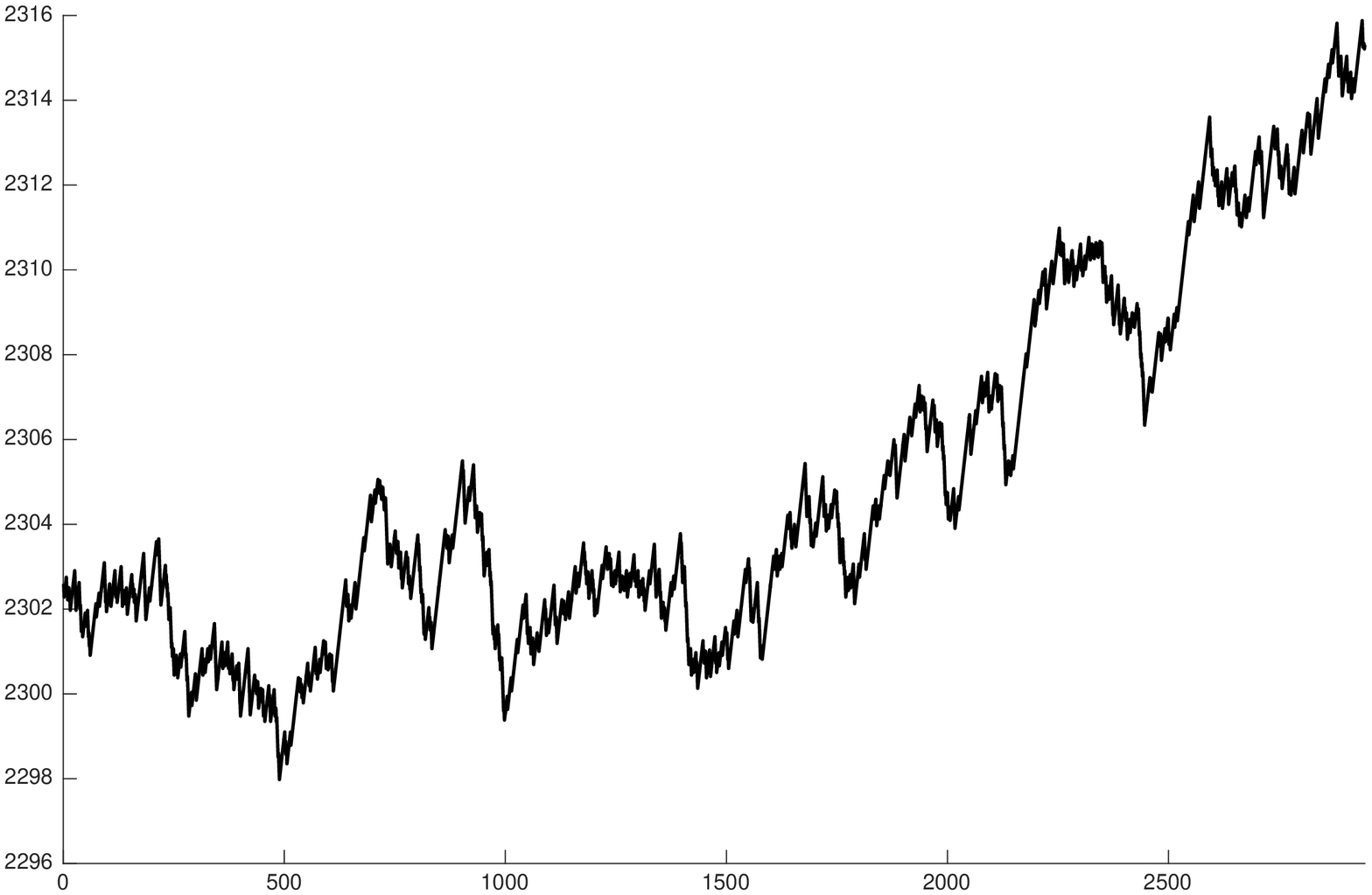}
	\caption{ $\log \big(x_i \bar\lambda^i \big)$ using the base 6 greedy algorithm for a random element in $A_{1496}$ (left) and one in $A_{2944}$ (right).}
	\label{fig:brownian}
\end{figure}

\section*{Acknowledgements}
The author thanks Stefan Steinerberger, Juan Arias de Reyna, and David Bevan for feedback on an earlier preprint.

The author is also grateful for access to the \textsc{Omega} cluster of the Yale University High Performance Computing Center.

\appendix
\section{Proof of Lemma \ref{lem:clt}}
\label{pf:clt}

\begin{lemma*}[\ref{lem:clt}]
	Let $\mathbf{\Phi}$ be a Markov chain on a finite space $X$. If there exists a state $i \in X$ and some $m \in \NN$ such that $i$ can be reached in exactly $m$ steps from any starting point with positive probability, that is,
$$P^m(x,i) > 0 \quad \mbox{ for all} \quad x \in X,$$
then any starting distribution approaches the unique stationary distribution $\pi$ at a geometric rate.

Moreover, we have the following: let $g:X \to \RR$ be bounded and 
	\[ S_n(g) = \sum_{k=1}^n g(\Phi_k). \]
Then if we let $\bar{g} = g - \int g\, d\pi$ where $\pi$ is the limit distribution, there exists a constant $0 \leq \gamma_g^2 < \infty$ such that
	\[ \lim_{n \to \infty} \Var \left[ \frac{1}{\sqrt{n}} S_n(\bar{g}) \right] = \gamma_g^2. \]

\end{lemma*}
\begin{proof}
We prove this using results from Meyn and Tweedie \cite{meyn}. An important concept is \emph{uniform ergodicity}: we say that a Markov chain $\mathbf{\Phi}$ is uniformly ergodic if there exist $r>1$ and $R < \infty$ such that for all $x$
	\[ \| P^n (x, \cdot) - \pi(\cdot) \| \leq R r^{-n}, \]
where $\pi$ is the stationary distribution, interpreted as a measure on $X$ (see \cite[Theorem 16.0.2.]{meyn}). Put differently, $P^n(x, \cdot)$ approaches $\pi(\cdot)$ geometrically. Here we define $\| \nu \|$ for a measure $\nu$ to be
	\[ \| \nu \| := \sup_{g \colon |g| \leq 1} \left| \int g(x) \nu(dx) \right|. \]
Uniform ergodicity is a special case of what is called $V$-uniform ergodicity (see \cite[Chapter 16]{meyn}) such that $V$ is a constant function. Theorem 17.0.1.ii in \cite{meyn} implies that this is enough to give the desired result for $\mathbf{\Phi}$. So if we can establish uniform ergodicity we have geometric convergence to the limit distribution and convergence of the variance of $n^{-1/2} S_n(g)$ to a finite constant depending on $g$. \\

The criterion we will use for showing ergodicity uses the notion of a small set (see \cite[Section 5.2.]{meyn}). If $\nu$ is a measure on the state space $X$, a measurable set $C$ is called \emph{$\nu$-small} if there exists $m \in \NN$ such that for all $x \in X$ and all measurable sets $B$
$$	P^m (x,B) \geq \nu(B).$$
Theorem 16.0.2 in \cite{meyn} states that a Markov chain is uniformly ergodic if and only if $X$ is $\nu$-small for some $\nu$.
This is quite easy to show for our purposes: suppose that $X$ is finite and some state $i \in X$ can be reached in exactly $m$ steps from any starting point. Then let
	\[ l = \min_{x \in X} P^m(x, i) > 0 \]
and define a measure
	\[ \nu(A) = \left\{\begin{array}{ll}
				l, & i \in B \\
				0, & i \not\in B
			\end{array}\right. .\]
Now if a set $B$ does not contain $i$ then its $\nu$-measure is 0 and the condition is trivially satisfied. Otherwise,
	\[ P^m(x, B) \geq P^m (x, \{i\}) \geq l = \nu (B). \]
Therefore $X$ is $\nu$-small and the Markov chain is uniformly ergodic.
\end{proof}

\section{Proof of Claim \ref{claim:R_onto}}
\label{sec:R_onto_pf}

First we prove the following:
\begin{lemma}
	Let $A$ be a base $a$ algorithm with dividing factor function $\Lambda$, and suppose the residue $j$ is a successor of the residue $i$: that is, there exists a natural number $n \equiv i \pmod{a}$ such that $n_1 \equiv j \pmod{a}$, where $n_1$ is the result of the first step of the algorithm $A$.
	
	Then if $\lambda = \Lambda(i)$ and $s = i \bmod \lambda$, there exists an integer $k$ such that
		\[ j = (ka + i - s)/\lambda. \]
\end{lemma}

\begin{proof}
	First note that $i \bmod \lambda = n \bmod \lambda$, because $i = n \bmod a$ and $\lambda \mid a$. So, by definition of the algorithm $A$,	
		\[ n = \lambda n_1 + s. \]
	Writing $n = la + i$ and $n_1 = ma + j$ for appropriate integers $l,m$, we get
		\[ la + i = \lambda(ma + j) + s. \]
	Rearranging, 
		\[ j = \big((l-\lambda m)a + i - s\big)/\lambda, \]
	so setting $k = \lambda m$ gives the desired expression.
\end{proof}

This result is used to prove the claim:
\begin{claim*}[\ref{claim:R_onto}]
	For every walk $\{v_0,v_1,\ldots,v_k\}$ on the graph $G(A)$ ending at the vertex $1$ (that is, $v_k=1$), there exists some natural number $n$ such that $R(n) = \{v_0,v_1, \ldots, v_{k}\}$.
\end{claim*}

\begin{proof}
	We prove by induction on $k$. The base case, $k=0$, is trivial as $R(1)=\{1\}$ gives the unique walk of length 1 ending at 1.
	
	Now suppose we have the result for $k-1$, and are given a walk $\{v_0, v_1, \ldots, v_k\}$. Using the inductive hypothesis, fix a natural number $n_1$ such that $R(n_1) = \{v_1,v_2, \ldots, v_k\}$. Then define $\lambda=\Lambda(v_0)$ and $s = v_0 \bmod \lambda$, and define
		\[ n = \lambda n_1 + s. \]
	We will show that $R(n) = \{v_0, v_1, \ldots, v_k\}$. By the definition of $n_1$, we can write
		\[ n_1 = m a + v_1 \]
	for some integer $m$, and since $v_1$ is a successor of $v_0$ we can write
		\[ v_1 = (ka + v_0 - s)/\lambda \]
	for some integer $k$. We then get
	\begin{align*}
		n &= \lambda(ma + v_1) + s \\
		 &= \lambda\big(ma + (ka + v_0 - s)/\lambda\big) + s \\
		 &= (\lambda m + k)a + v_0,
	\end{align*}
	so $n \bmod a = v_0$.
	
	It follows from this that the first residue of $R(n)$ is $\{v_0\}$, and the first step of the algorithm applied to $n$ will actually output
		\[ n = \lambda \cdot [ n_1 ] + s \]
	so, by choice of $n_1$, the remaining residues will be $\{v_1, v_2, \ldots, v_k\}$.
\end{proof}

\section{Proof of Selfridge's lower bound}
\label{sec:guyproof}
Here we give a prove of a result mentioned in \cite{guy1} and attributed to John Selfridge
because we were unable to find it in the literature and consider it worthwhile to
have it written. However, we also note that very similar (and more advanced) types of argument
along the same lines also appear in a paper by Altman \& Zelinsky \cite{alt1}.

\begin{theorem}[Selfridge] We have
	\[ f(n) \geq 3 \log_3 n, \]
with equality exactly when $n$ is a power of 3. 

\end{theorem} 

\begin{proof}
We use induction: suppose the bound holds for all numbers less than $n$. Use the formula
	\[ f(n) = \min \{ f(a) + f(b)\ :\ a,b < n \text{ such that } ab = n \text{ or } a+b = n \}. \]
First consider the case where $a$ or $b$ is 1; assume without loss of generality that $a=1$. Then
	\[ f(a) + f(b) = f(1) + f(n-1) \geq 1 + 3 \log_3 (n-1). \]
A straightforward calculation shows that this is greater than $3 \log_3 n$ for $n \geq 4$.

The other case is where $a,b \geq 2$. Then
	\[ f(a) + f(b) \geq 3 \log_3 a + 3 \log_3 b = 3 \log_3 ab. \]
If $ab = n$ then this satisfies the desired bound. Otherwise, suppose without loss of generality that $a \leq b$. Then
	\[ ab \geq 2b \geq a+b \]
so
	\[ 3 \log_3 ab \geq 3 \log_3 (a+b) = 3 \log_3 n. \]
The two cases considered together give a proof for $n \geq 4$; the remaining cases are easy to check.
\end{proof}

\section{Best known algorithm}
\label{app:alg}
The following is the $2310$-tuple representation of the current best known algorithm, due to David Bevan, which gives the bound in Theorem \ref{maintheorem}.
\begin{center}
{\fontsize{6pt}{1em}\selectfont
\begin{verbatim}
(3,  3,  2,  3,  2,  5,  3,  7,  2,  3,  2, 11,  2,  6,  2,  3,  2,  2,  3,  3,  2,  3,  2,  2,  2,  2,  2,  3,  2,  2
3,  3,  2,  3,  2,  7,  6,  6,  2,  3,  2,  2,  3,  3,  2,  3,  2,  2,  3,  7,  2,  3,  2,  2,  3, 55,  2,  3,  2,  2
3,  3,  2,  3,  2,  5,  3,  3,  2,  3,  2,  2,  2,  2,  2,  3,  2,  7,  3,  3,  2,  3,  2,  2,  6,  6,  2,  3,  2,  2
3,  7,  2,  3,  2,  5,  3,  3,  2,  3,  2,  2,  3,  3,  2,  3,  3,  3,  3,  3, 55,  3,  2,  2,  3,  5,  2,  3,  2,  7
3,  2,  2,  3,  2,  5,  3,  3,  2,  3,  2,  2,  3,  7,  2,  3,  2,  2,  3,  3,  2,  3,  2, 11,  6,  6,  2,  3,  2,  2
3,  3,  2,  3,  2,  5,  3,  3,  2,  3,  2,  7,  3,  3,  2,  3,  3,  3,  3,  3,  2,  3,  2,  2,  3,  7,  2,  3,  2,  2
3,  3,  2,  3,  2,  5,  3,  3,  2,  3,  2,  2,  3,  3,  2,  3,  2,  2,  3,  3,  2,  3,  2,  7,  3,  3,  2,  3,  2, 11
3,  3,  2,  3,  2,  5,  2,  7,  2,  3,  2,  2,  3,  3,  2,  3,  2,  2,  2,  2,  2,  3,  2,  2,  3,  3,  2,  3,  7,  7
3,  3,  2,  3,  2,  5,  3,  3,  2,  3,  2,  2,  6,  6,  2,  3,  3,  3,  3,  7,  2,  3,  2,  2,  3,  2,  2,  3,  2,  2
3,  3,  2,  3,  2,  5,  3,  3,  2,  3,  2,  2,  2,  2,  2,  3,  2,  7,  2,  2,  2,  3,  2,  2,  3,  3,  2,  3,  2,  2
3,  3,  2,  3,  2,  5,  3,  3,  2,  3,  2,  2,  3,  3,  2,  3,  3,  3,  3,  3,  2,  3,  7,  7,  3,  3,  2,  3,  2,  7
3,  3,  2,  3,  2,  5,  3,  3,  2,  3,  2, 11,  3,  7,  2,  3,  2,  2,  3,  3,  2,  3,  2,  2,  3,  3,  2,  3,  2,  2
3,  3,  2,  3,  7,  5,  3,  3,  2,  3,  2,  2,  3,  3,  2,  3,  2,  2,  3,  3,  2,  3,  2,  2,  2, 55, 55,  3,  2,  2
3,  3,  2,  3,  2,  5,  3,  3,  2,  3,  2,  2,  3,  3,  2,  3,  2, 11,  3,  3,  2,  3,  2,  7,  3,  3,  2,  3,  2,  2
3,  3,  2,  3,  2,  5,  2,  7,  2,  3,  2,  2,  6,  6,  2,  3,  2,  2,  3,  3,  2,  3,  2,  2,  3,  5,  2,  3,  2,  2
3,  3,  2,  3,  2,  7,  3,  3,  2,  3,  2,  2,  3,  3,  2,  3,  2,  2,  3,  7,  2,  3,  2, 11,  3,  3,  2,  3,  2,  2
3,  3,  2,  3,  2,  5,  3,  3,  2,  3,  2,  2,  3,  2,  2,  3,  3,  7,  3,  3,  2,  3,  2,  2,  3,  3,  2,  3,  2,  2
3,  3,  2,  3,  2,  5,  3,  2,  2,  3,  2,  2,  3,  3,  2,  3,  3,  3,  3,  3,  2,  3,  7,  7,  3,  3,  2,  3,  2,  7
3,  3,  2,  3,  2,  5,  3,  3,  2,  3,  2,  2,  3,  7,  2,  3,  2,  2,  3,  3,  2,  3,  2,  2,  3,  3,  2,  3,  2,  2
3,  3,  2,  3,  2,  5,  3,  3,  2,  3,  2,  2,  3,  2,  2,  3,  3,  3,  3,  3,  2,  3,  2,  2,  3,  7,  2,  3,  2,  2
3,  3,  2,  3,  2,  5,  3,  3,  2,  3,  2,  2,  3,  3,  2,  3,  2,  2,  3,  3,  2,  3,  2,  7,  3,  3,  2,  3,  2,  2
3,  3,  2,  3,  2,  5,  3,  7,  2,  3,  2,  2,  3,  3,  2,  3,  2,  2,  2,  2,  2,  3,  2,  2,  3,  3,  2,  3,  2,  2
3,  3,  2,  3,  2,  5,  3,  3,  2,  3,  2, 11,  3,  3,  2,  3,  2,  2,  3,  7,  2,  3,  2,  2,  2,  2,  2,  3,  2,  2
3,  3,  2,  3,  3,  5,  3,  3,  2,  3,  2,  2,  3,  3,  2,  3,  2,  7,  3,  3,  2,  3,  2,  2,  3, 55,  2,  3,  2,  2
3,  3,  2,  3,  2,  5,  3,  6,  2,  3,  2,  2,  3,  3,  2,  3,  3, 11,  3,  3,  2,  3,  2,  2,  3,  3,  2,  3,  2,  7
3,  3,  2,  3,  2,  5,  6,  6,  2,  3,  2,  2,  3,  7,  2,  3,  2,  2,  3,  3,  2,  2,  2,  2,  3,  5,  2,  7,  2,  2
3,  2,  2,  3,  2,  5,  3,  3,  2,  3,  2,  7,  3,  3,  2,  3,  3,  3,  3,  2,  2,  3,  2, 11,  3,  7,  2,  3,  2,  2
3,  3,  2,  3,  2,  5,  3,  3,  2,  3,  2,  2,  3,  3,  2,  3,  3,  3,  3,  3,  2,  3,  2,  7,  3,  3,  2,  3,  2,  2
3,  3,  2,  3,  2,  5,  2,  7,  2,  3,  2,  2,  3,  3,  2,  3,  2,  2,  3,  3,  2,  3,  2,  2,  2,  2,  2,  3,  2, 11
3,  3,  2,  3,  2,  5,  3,  3,  2,  3,  2,  2,  3,  3,  2,  3,  2,  2,  3,  7,  2,  3,  2,  2,  3,  3,  2,  3,  2,  2
3,  3,  2,  3,  2,  5,  3,  3,  2,  3,  2,  2,  2,  2,  2,  3,  2,  7,  3,  3,  2,  3,  2,  2,  3,  3,  2,  3,  2,  2
3,  7,  2,  3,  2,  5,  3,  3,  2,  3,  2,  2,  3,  3,  2,  3,  2,  2,  3,  3,  2,  3,  2,  2,  3,  3,  2,  3,  2,  7
3,  3,  2,  3,  2,  5,  3,  3,  2,  3,  2,  2,  3,  7,  2,  3,  2,  2,  3,  3,  2,  3,  2,  2,  3,  3,  2,  3,  2,  2
3,  3,  2,  3,  2,  5,  3,  3,  2,  3,  2, 11,  3,  3,  2,  3,  2,  2,  3,  3,  2,  3,  2,  2,  3,  7,  2,  3,  2,  2
3,  3,  2,  3,  2,  5,  3,  3,  2,  3,  2,  2,  3,  3,  2,  3,  2,  2,  3,  3,  2,  3,  2,  7,  3, 55,  2,  3,  2,  2
3,  3,  2,  3,  2,  5,  3,  7,  2,  3,  2,  2,  3,  3,  2,  3,  2, 11,  3,  3,  2,  3,  2,  2,  2,  2,  2,  3,  2,  2
3,  3,  2,  3,  2,  5,  3,  3,  2,  3,  2,  2,  3,  3,  2,  3,  2,  2,  3,  7,  2,  3,  2,  2,  3,  3,  2,  3,  2,  2
3,  3,  2,  3,  2,  5,  3,  3,  2,  3,  2,  2,  3,  3,  2,  3,  2,  7,  2,  2,  2,  3,  2, 11,  3,  3,  2,  3,  2,  2
3,  2,  2,  3,  2,  5,  2,  2,  2,  3,  2,  2,  2,  2,  2,  3,  3,  3,  3,  3,  5,  3,  7,  7,  3,  3,  2,  3,  2,  7
3,  3,  2,  3,  2,  5,  3,  3,  2,  3,  2,  2,  3,  7,  2,  3,  3,  3,  3,  3,  2,  3,  2,  2,  3,  3,  2,  3,  2, 11
3,  3,  2,  3,  7,  5,  3,  3,  2,  3, 55, 55,  2,  2,  2,  3,  2,  2,  3,  3,  2,  3,  2,  2,  3,  7,  2,  3,  2,  2
3,  3,  2,  3,  2,  5,  3,  3,  2,  3,  2,  2,  3,  3,  2,  3,  2,  2,  3,  3,  2,  3,  2,  7,  3,  2,  2,  3,  2,  2
3,  3,  2,  3,  2,  5,  3,  7,  2,  3,  2,  2,  3,  3,  2,  3,  2,  2,  3,  2,  2,  3,  2,  2,  3,  3,  2,  3,  2,  2
3,  3,  2,  3,  2,  5,  6,  6,  2,  3,  2,  2,  3,  3,  2,  3,  3,  3,  3,  7,  2,  3,  2,  2,  3,  3,  2,  3,  2,  2
3,  3,  2,  3,  2,  5,  2,  2,  2,  3,  2, 11,  3,  3,  2,  3,  2,  7,  2,  2,  2,  3,  2,  2,  3,  3,  2,  3,  2,  2
3,  2,  2,  3,  2,  5,  3,  3,  2,  3,  2,  2,  3,  3,  2,  3,  3,  3,  3,  3,  2,  3,  7,  7,  3, 55,  2,  3,  2,  7
3,  3,  2,  3,  2,  5,  3,  3,  2,  3,  2,  2,  3,  7,  2,  3,  2, 11,  3,  3,  2,  3,  2,  2,  3,  3,  2,  3,  2,  2
3,  3,  2,  3,  7,  5,  3,  3,  2,  3,  2,  2,  3,  3,  2,  3,  3,  3,  3,  3, 55,  3,  2,  2,  3,  7,  2,  3,  2,  2
3,  6,  2,  3,  2,  5,  3,  3,  2,  3,  2,  2,  3,  3,  2,  3,  2,  2,  3,  3,  2,  3,  2,  7,  3,  3,  2,  3,  2,  2
3,  3,  2,  3,  2,  5,  3,  7,  2,  3,  2,  2,  3,  3,  2,  3,  3,  3,  3,  3,  2,  3,  2,  2,  3,  3,  2,  3,  7,  7
3,  3,  2,  3,  2,  5,  3,  3,  2,  3,  2,  2,  3,  3,  2,  3,  3,  3,  3,  7,  2,  3,  2,  2,  3,  2,  2,  3,  2, 11
3,  3,  2,  3,  2,  5,  2,  2,  2,  3,  2,  2,  2,  2,  2,  3,  2,  7,  3,  3,  2,  3,  2,  2,  7,  7,  2,  3,  2,  2
3,  3,  2,  3,  2,  5,  3,  3,  2,  3,  2,  2,  3,  6,  2,  3,  3,  3,  3,  3,  2,  3,  7,  7,  3,  2,  2,  3,  2,  7
3,  3,  2,  3,  2, 55,  3,  3,  2,  3,  2,  2,  3,  7,  2,  3,  2,  2,  3,  3,  2,  3,  2,  2,  3,  3,  2,  3,  3,  3
3,  3,  2,  3,  2,  5,  3,  3,  2,  3,  2,  2,  3,  3,  2,  3,  2,  2,  3,  3,  2,  3,  2,  2,  3,  7,  2,  3,  2,  2
3,  3,  2,  3,  2,  5,  3,  3,  2,  3,  2, 11,  3,  3,  2,  3,  2,  2,  2,  2,  2,  3,  2,  2,  3,  3,  2,  3,  2,  2
3,  3,  2,  3,  2,  5,  3,  7,  2,  3,  2,  2,  3,  3,  2,  3,  2,  2,  3,  3,  2,  3,  2,  2,  2, 55,  2,  3,  2, 14
3,  3,  2,  3,  2,  5,  3,  3,  2,  3,  2,  2,  3,  3,  2,  3,  3, 11,  3,  7,  2,  3,  2,  2,  3,  2,  2,  3,  2,  2
3,  3,  2,  3,  2,  5,  3,  3,  2,  3,  2,  2,  3,  3,  2,  3,  2,  7,  3,  3,  2,  3,  2,  2,  3,  3,  2,  3,  2,  2
3,  3,  2,  3,  2,  5,  3,  3,  2,  3,  2,  2,  3,  3,  2,  3,  3,  3,  3,  3,  2,  3,  2, 11,  3,  3,  2,  3,  2,  7
3,  3,  2,  3,  2,  5,  3,  3,  2,  3,  2,  2,  3,  7,  2,  3,  3,  3,  3,  3,  2,  3,  2,  2,  3,  3,  2,  3,  2,  2
3,  3,  2,  3,  2,  5,  3,  3,  2,  3,  2,  2,  3,  3,  2,  3,  3,  3,  3,  3,  2,  3,  2,  2,  3,  7,  2,  3,  2, 11
3,  3,  2,  3,  2,  5,  3,  3,  2,  3, 55, 55,  3,  3,  2,  3,  2,  2,  3,  3,  2,  3,  2,  7,  3,  3,  2,  3,  2,  2
6,  3,  2,  3,  2,  5,  3,  7,  2,  3,  2,  2,  3,  3,  2,  3,  2,  2,  3,  3,  2,  3,  2,  2,  3,  3,  2,  3,  2,  2
3,  3,  2,  3,  2,  5,  3,  3,  2,  3,  2,  2,  3,  3,  2,  3,  2,  2,  3,  7,  2,  3,  2,  2,  2,  2,  2,  3,  2,  2
3,  3,  2,  3,  2,  5,  3,  3,  2,  3,  2,  2,  3,  3,  2,  3,  2,  7,  3,  3,  2,  3,  2,  2,  3,  3,  2,  3,  2,  2
3,  3,  2,  3,  2,  5,  2,  2,  2,  3,  2, 11,  3,  3,  2,  3,  3,  3,  3,  3,  2,  3,  2,  2,  3,  3,  2,  3,  2,  7
3,  3,  2,  3,  2,  5,  3,  3,  2,  3,  2,  2,  3,  7,  2,  3,  2,  2,  3,  3,  2,  3,  2,  2,  3, 55,  2,  3,  2,  2
3,  3,  2,  3,  2,  5,  3,  3,  2,  3,  2,  2,  2,  2,  2,  3,  3, 11,  3,  3,  2,  3,  2,  2,  3,  7,  2,  3,  2,  2
3,  3,  2,  3,  2,  5,  3,  3,  2,  3,  2,  2,  3,  3,  2,  3,  7,  7,  3,  3,  2,  3,  2,  2,  3,  3,  2,  3,  2,  2
3,  2,  2,  3,  2,  5,  3,  7,  2,  3,  2,  2,  3,  3,  2,  3,  2,  2,  3,  3,  2,  3,  2, 11,  3,  3,  2,  3,  2,  2
3,  3,  2,  3,  2,  5,  3,  3,  2,  3,  2,  2,  3,  3,  2,  3,  3,  3,  3,  7,  2,  3,  2,  2,  3,  3,  2,  3,  2,  2
3,  3,  2,  3,  2,  5,  3,  3,  2,  3,  2,  2,  3,  3,  2,  3,  2,  7,  3,  3,  2,  3,  2,  2,  3,  3,  2,  3,  2, 11
3,  2,  2,  3,  2,  5,  3,  3,  2,  3,  2,  2,  2,  2,  2,  3,  3,  3,  3,  3,  2,  3,  2,  2,  3,  3,  2,  3,  2,  7
3,  3,  2,  3,  2,  5,  3,  3,  2,  3,  2,  2,  3,  7,  2,  3,  3,  3,  3,  3,  2,  3,  2,  2,  3,  3,  2,  3,  2,  2
3,  3,  2,  3,  2,  5,  3,  3,  2,  3,  2,  2,  3,  3,  2,  3,  2,  2,  6,  6,  2,  3,  2,  2,  2,  7,  2,  3,  2,  2
3,  3,  2,  3,  2,  5,  3,  3,  2,  3,  2,  2,  3,  3,  2,  3,  2,  2,  2,  2,  2,  3,  2,  7,  3,  3,  2,  3,  2,  2)
\end{verbatim}
}
\end{center}


\begin{thebibliography}{}

\bibitem{alt} H. Altman, 
Integer complexity and well-ordering. 
Michigan Math. J. 64 (2015), no. 3, 509-538. 

\bibitem{alt2} H. Altman, 
Integer complexity: algorithms and computational results.
arXiv:1606.03635

\bibitem{alt1} H. Altman, J. Zelinsky, 
Numbers with integer complexity close to the lower bound. 
Integers 12 (2012), no. 6, 1093-1125. 

\bibitem{reyna} J. Arias de Reyna and J. van de Lune, Algorithms for determining integer complexity. arXiv:1404.2183v2

\bibitem{latvia2} J. \u{C}er\c{n}enoks, J. Iraids, M. Opmanis, R. Opmanis, and K. Podnieks, Integer Complexity: Experimental and Analytical Results II. arxiv:1409.0446v1

\bibitem{guy1} R. K. Guy, Unsolved Problems: Some Suspiciously Simple Sequences. Amer. Math. Monthly 93 (1986), no. 3, 186--190. 

\bibitem{guy2} R. K. Guy, Unsolved Problems in Number Theory, $3^{rd}$ edition, Springer, 2004.

\bibitem{latvia1} J. Iraids, K. Balodis, J. \u{C}er\c{n}enoks, M. Opmanis, R. Opmanis, and K. Podnieks, Integer Complexity: Experimental and Analytical Results. arXiv:1203.6462

\bibitem{mahler} K. Mahler and J. Popken, On a maximum problem in arithmetic. (Dutch) Nieuw Arch. Wiskunde (3) 1, (1953). 1--15. 

\bibitem{meyn} S.P. Meyn and R.L. Tweedie, Markov chains and stochastic stability. Springer-Verlag, London (1993). Available at: probability.ca/MT 

\bibitem{sinai1} Ya. Sinai, Statistical ($3x+1$) problem. Dedicated to the memory of J\"urgen K. Moser. Comm. Pure Appl. Math. 56 (2003), no. 7, 1016-1028. 

\bibitem{sinai2} Ya. Sinai, Uniform distribution in the ($3x+1$)-problem. Mosc. Math. J. 3 (2003), no. 4, 1429-1440. 

\bibitem{sinai3} Ya. Sinai, A theorem about uniform distribution. Comm. Math. Phys. 252 (2004), no. 1-3, 581-588. 

\bibitem{steiner} S. Steinerberger, A Short Note on Integer Complexity. Contributions to Discrete Mathematics 9 (2014), 63--69.

\end{thebibliography}
\end{document}